\documentclass{article}
\usepackage[a4paper,hscale=0.7,vscale=0.75,centering]{geometry}
\usepackage{mathtools}
\usepackage{amsthm,amssymb}
\usepackage{mathrsfs}

\newcommand{\bD}{\mathbb{D}}
\newcommand{\E}{\mathop{{}\mathbb{E}}\nolimits}
\newcommand{\cF}{\mathscr{F}}
\newcommand{\enne}{\mathbb{N}}
\renewcommand{\P}{\mathbb{P}}
\newcommand{\erre}{\mathbb{R}}

\newcommand{\cL}{\mathscr{L}}

\DeclareMathOperator{\dom}{\mathsf{D}}
\newcommand{\embed}{\hookrightarrow}

\newtheorem{prop}{Proposition}[section]
\newtheorem{thm}[prop]{Theorem}

\newtheorem{lemma}[prop]{Lemma}

\theoremstyle{definition}
\newtheorem{hyp}{Hypothesis}
\theoremstyle{remark}
\newtheorem{rmk}[prop]{Remark}
\newtheorem{example}[prop]{Example}

\DeclarePairedDelimiter\abs{\lvert}{\rvert}
\DeclarePairedDelimiter\norm{\lVert}{\rVert}
\DeclarePairedDelimiterX\ip[2]{\langle}{\rangle}{#1,#2}

\begin{document}

\title{Absolute continuity of solutions to reaction-diffusion
  equations with multiplicative noise}

\author{Carlo Marinelli\thanks{Department of Mathematics, University
    College London, Gower Street, London WC1E 6BT, UK. URL:
    \texttt{http://goo.gl/4GKJP}} \and Llu\'is
  Quer-Sardanyons\thanks{Departament de Matem\`atiques, Universitat
    Aut\`onoma de Barcelona, 08193 Cerdanyola del Vall\`es
    (Barcelona), Catalonia, Spain.}}

\date{May 21, 2019}

\maketitle

\begin{abstract}
  We prove absolute continuity of the law of the solution, evaluated
  at fixed points in time and space, to a parabolic dissipative
  stochastic PDE on $L^2(G)$, where $G$ is an open bounded domain in
  $\erre^d$ with smooth boundary. The equation is driven by a
  multiplicative Wiener noise and the nonlinear drift term is the
  superposition operator associated to a real function which is
  assumed to be monotone, locally Lipschitz continuous, and growing
  not faster than a polynomial. The proof, which uses arguments of the
  Malliavin calculus, crucially relies on the well-posedness theory in
  the mild sense for stochastic evolution equations in Banach spaces.

  \medskip


  \medskip

\end{abstract}

\maketitle

\section{Introduction}
Let $G$ be a bounded domain of $\erre^d$, $d>1$, with smooth
boundary. Consider a semilinear stochastic equation of the type
\begin{equation}
  \label{eq:0}
  du(t) + Au(t)\,dt = f(u(t))\,dt + \sigma(u) B\,dW(t),
  \qquad u(0)=u_0,
\end{equation}
where $A$ is the negative generator of an analytic semigroup on
$L^q(G)$, $q \geq 2$, $f:\erre \to \erre$ is a locally
Lipschitz continuous decreasing function with polynomial growth,
$\sigma: \erre \to \erre$ is a Lipschitz continuous function, $B$ is a
$\gamma$-Radonifying operator from $L^2(G)$ to $L^q(G)$, and $W$ is a
cylindrical Wiener process on $L^2(G)$ (precise assumptions on the
data of the problem are provided in \S\ref{sec:WP} below). Then
\eqref{eq:0} admits a unique mild solution which is continuous in
space and time.
Our aim is to prove that the law of the random variable $u(t,x)$ is
absolutely continuous with respect to Lebesgue measure for every
fixed $(t,x) \in \erre_+ \times G$.
It seems that, somewhat surprisingly, this natural question has not
been addressed in the literature. In fact, all results of which we are
aware about existence (and regularity) of the density of solutions to
SPDEs with multiplicative noise deal with the case where $G$ is the
whole space, $-A$ is the Laplacian, and the drift coefficient $f$ is
(globally) Lipschitz continuous (see, e.g.,
\cite{Marquez-Mellouk-SarraSPA2001,nualart-LNM,Nualart-QuerPOTA,Sanz-libro}
and references therein). Our results do not rely on any one of these
assumptions. In particular, we essentially just assume that the
semigroup generated by $-A$ is self-adjoint and given by a family of
kernel operators, so that, for instance, very large classes of
elliptic second-order operators are allowed, and the function $f$ can
be of polynomial type. Another major difference with respect to the
above-mentioned works is that we rely almost exclusively on the
interpretation of \eqref{eq:0} as an equation for an $L^q(G)$-valued
process, and that we view the pointwise Malliavin derivative of its
solution as a process taking values in $L^q(G;H)$, where $H$ is a
suitably chosen Hilbert space. This point of view, which allows us to
rely on powerful techniques of the functional-analytic approach to
stochastic evolution equations on UMD Banach spaces, is probably the
most interesting aspect of this work. The more common random field
interpretation of \eqref{eq:0}, that seems the only one used in
previous work, at least in connection with techniques of the Malliavin
calculus, is used here very sparingly, essentially only to take the
pointwise Malliavin derivative of the solution to \eqref{eq:0}.

Existence and regularity of the density of solutions to semilinear
heat equations with additive noise, i.e. for the easier case where
$\sigma=1$ and $-A$ is the Laplacian, were obtained in
\cite{Marinelli-Nualart-Quer}. Those results, however, depend heavily
on the noise being additive, and cannot be extended to the general
setting considered here. In fact, if the noise is additive, then the
Malliavin derivative of the solution satisfies a deterministic
equation with random coefficients, which yields quite strong estimates
using pathwise arguments. On the other hand, if the noise is
multiplicative, then the Malliavin derivative is only expected to
satisfy a further stochastic evolution equation with quite singular
initial condition, which is much more difficult to handle than the
deterministic PDE arising in the case of additive noise. As a
consequence, while in \cite{Marinelli-Nualart-Quer} we obtained
existence as well as regularity of the density, here we can only show
existence. As it is natural to expect, regularity could be obtained
also in the case of multiplicative noise and Lipschitz continuous
drift. However, we concentrate here only on the existence issue, and
we shall deal with the regularity problem somewhere else, hopefully
also in the general case where $f$ is monotone and polynomially
bounded.

Let us briefly describe the main content of the paper. We first show
existence and uniqueness of a unique mild solution $u$ to \eqref{eq:0}
which is continuous in space and time. This follows by relatively
recent results on well-posedness in the mild sense for stochastic
evolution equations in Banach spaces (see \S\ref{sec:WP}). Assuming
that the semigroup generated by $-A$ is a family of kernel operators,
the mild solution can be interpreted also in the sense of random
fields.
Considering first the case where $f$ is Lipschitz continuous, so that
the mild solution is the unique fixed point of an operator $\Phi$,
this reformulation allows to compute the Malliavin derivative of
$\Phi$ applied to a class of sufficiently regular processes. Using
estimates for stochastic convolutions in Banach spaces, we show that
the fixed-point operator $\Phi$ leaves invariant a subspace of
Malliavin differentiable processes with finite moment. This yields, by
closability properties of the Malliavin derivative, that the unique
mild solution to \eqref{eq:0} is pointwise Malliavin
differentiable. As a second step, we provide sufficient conditions
ensuring that the Malliavin derivative is non-degenerate, adapting a
method used in \cite[theorem~5.2]{Nualart-QuerPOTA} for equations on
$\erre^d$ (see \S\ref{sec:Lip}). This yields, as is well known, the
pointwise absolute continuity of the law of the solution.  As
mentioned above, the results should be interesting in their own right,
as equations in domains (in dimension higher than one) do not appear
to have been considered in the literature.
Finally, in the general case of equations of reaction-diffusion type,
the pointwise absolute continuity of the law of the solution is
treated by localization techniques, i.e. by means of the
Bouleau-Hirsch criterion (see \S\ref{sec:loc}), and by convergence
results for stochastic evolution equations with locally
Lipschitz continuous coefficients in spaces of continuous functions.

\medskip

\noindent\textbf{Acknowledgments.} The first-named author is sincerely
grateful to Prof. S.~Albeverio for several very pleasant stays at the
Interdisziplin\"ares Zentrum f\"ur Komplexe Systeme, Universit\"at
Bonn, where most of the work for this paper was done.
The second-named author is supported by the grant MTM2015-67802P.


\section{Well-posedness in the space of continuous functions}
\label{sec:WP}
We are going to establish well-posedness in the mild sense for the
stochastic equation \eqref{eq:0} in a space of continuous functions,
using general well-posedness results for stochastic evolution
equations in UMD Banach spaces (see \cite{KvN2,vNVW}).
Assuming that the semigroup generated by $-A$ is a family of integral
operators, we shall also show that the solution thus obtained can be
viewed as a solution in the sense of random field (cf.
\cite{Dalang-EJP1999, Walsh-LNM}).

\subsection{Preliminaries}
Let us consider the following stochastic evolution equation, posed on
a general Banach space $X$:
\begin{equation}
  \label{eq:EE}
  du(t) + Au(t)\,dt = f(u(t))\,dt + B(u(t))\,dW(t),
  \qquad u(0)=u_0,
\end{equation}
where $W$ is a cylindrical Wiener process on a Hilbert space $U$, and
all other coefficients are specified below. The following
well-posedness result is a slightly simplified version of
\cite[theorem~4.9]{KvN2}.
\begin{thm}
  \label{thm:astro}
  Let $E$ be a UMD Banach space with type 2, such that $X$ is densely
  and continuously embedded in $E$ densely, and $A$ be a sectorial,
  accretive operator on $E$ such that the semigroup $S$ on $E$
  generated by $-A$ restricts to a $C_0$-semigroup of contractions on
  $X$. Assume that $f: X \to X$ is locally Lipschitz continuous and
  there exists $m>0$ such that
  \begin{align*}
    \ip[\big]{f(x+y)-f(y)}{x^*}
    &\lesssim 1 + \norm{y}^m - \norm{x}^m,\\[6pt]
    \norm[\big]{f(y)}
    &\lesssim 1 + \norm{y}^m
  \end{align*}
  for all $x$, $y \in X$ and $x^* \in \partial\norm{x}$.
  Let $p>2$ and assume that there exist numbers $\eta \in \erre_+$,
  with
  \[
  \eta < \frac12  - \frac{1}{p},
  \]
  such that $E_\eta := \dom((I+A)^\eta)$ is densely and continuously
  embedded in $X$.
  If $B: X \to \gamma(U,E)$ is locally Lipschitz continuous with
  linear growth, and $u_0 \in L^p(\Omega;X)$, then there exists a
  unique $X$-valued mild solution to \eqref{eq:EE}, which satisfies
  \[
  \E \sup_{t \leq T} \norm{u(t)}_X^p \lesssim 1 + \E\norm{u_0}_X^p.
  \]
\end{thm}
Here $\partial\norm{x}$ stands for the subdifferential at $x$, in the
sense of convex analysis, of the convex function
$\norm{\cdot}$, that is, denoting the dual of $X$ by $X'$,
\[
  \partial\norm{x} = \bigl\{ x^* \in X':\, \norm{x^*}=1, \; \ip{x^*}{x}=1
  \bigr\}.
\]
Moreover, the notation $a \lesssim b$ means that there
exists a constant $N$ such that $a \leq Nb$. To emphasize the
dependence of $N$ on parameters $p_1,\ldots,p_n$, we shall write
$a \lesssim_{p_1,\ldots,p_n}
b$. 

\begin{rmk}
  In \cite{KvN2} the authors also require that
  \[
    \ip[\big]{-Ax + f(x+y)}{x^*} \lesssim 1 + \norm{y}^m + \norm{x}
  \]
  for every $x \in \dom(A\vert_X)$ and $x,y \in X$.  Since we are
  assuming that $A$ is accretive in $X$, it follows that
  $\ip{-Ax}{x^*} \leq 0$. Moreover,
  \begin{align*}
    \ip[\big]{f(x+y)}{x^*}
    &= \ip[\big]{f(x+y)-f(y)}{x^*} + \ip[\big]{f(y)}{x^*}\\
    &\lesssim 1 + \norm{y}^m + \abs[\big]{\ip[\big]{f(y)}{x^*}}
      \lesssim 1 + \norm{y}^m,
  \end{align*}
  hence their condition, under our assumptions, is automatically
  satisfied.
\end{rmk}

\begin{rmk}
  Further well-posedness results in $L^q$ spaces for semilinear
  parabolic SPDEs of accretive type, with more natural assumptions on
  the nonlinear drift term $f$, can be found in
  \cite{cm:Ascona13,cm:SIMA18,cm:ref,cm:semimg,cm:AP18}. See also
  \cite{cerrai03} for related results in spaces of continuous
  functions.
\end{rmk}

\medskip

We shall also need some basic facts on interpolation. The real and the
complex interpolation functors are denoted by $(\cdot,\cdot)$ and
$[\cdot,\cdot]$, respectively. Moreover, we shall write $X \embed Y$
to mean that $X$ is continuously embedded in $Y$.
\begin{lemma}
  \label{lm:interp}
  Let $X$ and $Y$ be two Banach spaces forming an interpolation pair,
  $A$ a positive operator on $X$, and
  $\theta,\theta' \in \mathopen]0,1\mathclose[$, $q,q' \in [1,\infty]$ be
  constants. The following statements hold true:
  \begin{itemize}
  \item[\emph{(a)}] if $X \supset Y$, then
    $(X,Y)_{\theta,q} \embed (X,Y)_{\theta',q'}$;
  \item[\emph{(b)}] $(X,Y)_{\theta,1} \embed (X,Y)_{\theta,\infty}$;
  \item[\emph{(c)}]
    $(X,Y)_{\theta,1} \embed [X,Y]_\theta \embed
    (X,Y)_{\theta,\infty}$;
  \item[\emph{(d)}]
    $\bigl( X, \dom(A) \bigr)_{\theta,1} \embed \dom(A^\theta) \embed
    \bigl( X, \dom(A) \bigr)_{\theta,\infty}$.
  \end{itemize}
\end{lemma}
\begin{proof}
  All statements can be found in \cite{TriebelI}. Specific references
  are provided for each result: (a) and (b) are parts of
  theorem~1.3.3, p.~25; (c) is a consequence of theorem 1, p.~64,
  taking into account definition~1.10.1, p.~61; (d) is part of theorem
  1.15.2, p.~101.
\end{proof}

\subsection{Existence of a unique mild solution}
Let us now turn to equation \eqref{eq:0}, about which the following
standing assumptions are assumed from now on.
\begin{hyp}
  \label{hyp:1}
  (a) The operator $A$ is the realization on $L^q(G)$, $q \geq 2$, of
  a second-order strongly elliptic operator with $C^\infty$
  coefficients, with Dirichlet boundary conditions.
  (b) The function $f:\erre \to \erre$ is an odd polynomial of degree
  $m>0$ with negative leading coefficient.
  (c) $W$ is a cylindrical Wiener process on $L^2(G)$ defined on a
  filtered probability space $(\Omega,\cF,(\cF_t)_{t \in [0,T]},\P)$,
  with $T \in \erre_+$, where $(\cF_t)_{t \in [0,T]}$ is the
  completion of the filtration generated by $W$.
\end{hyp}
\noindent It follows by (b) that $\abs{f(x)} \lesssim 1 + \abs{x}^m$ for all
$x \in \erre$.

\begin{prop}\label{prop:0}
  Assume that
  \[
  \frac{d}{2q} < \frac12 - \frac{1}{p},
  \]
  $\sigma \colon \erre \to \erre$ is locally Lipschitz continuous with
  linear growth, and $B \in \gamma(L^2(G),L^q(G))$. If
  $u_0 \in L^p(\Omega;C(\overline{G}))$, then \eqref{eq:0} admits a
  unique $C(\overline{G})$-valued mild solution $u$, which satisfies
  the estimate
  \[
    \E \sup_{t \leq T} \norm[\big]{u(t)}_{C(\overline{G})}^p \lesssim
    1 + \E\norm[\big]{u_0}_{C(\overline{G})}^p.
  \]
\end{prop}
\noindent Here $C(\overline{G})$ denotes the space of continuous functions on
$\overline{G}$, the closure of $G$.
\begin{proof}
  We are going to verify that the assumptions of
  theorem~\ref{thm:astro} are satisfied. It follows from hypothesis
  (A) that, for any $q \geq 2$, $A$ is a sectorial, accretive operator
  on $L^q(G)$, and that the semigroup $S$ generated by $-A$ restricts
  to a $C_0$-semigroup on $C(\overline{G})$ (see, e.g.,
  \cite[theorem~3.5, pp.~213-214 and theorem~3.7,
  p.~217]{Pazy}). Moreover, denoting the evaluation operator on
  $C(\overline{G})$ associated to $f$ by the same symbol, it is not
  difficult to see that $f$ satisfies the assumptions of
  theorem~\ref{thm:astro} (detail can be found in \cite[examples~4.2
  and 4.5]{KvN2}). Moreover, one easily verifies that
  $u \mapsto \sigma(u) B$ is locally Lipschitz continuous and has
  linear growth as a map from $C(\overline{G})$ to $\gamma(U,L^q(G))$.

  Let $\theta'<\theta$ be such that
  \[
  \frac{d}{2q} < \theta' < \theta < \frac12 - \frac{1}{p}.
  \]
  Setting $E:=L^q:=L^q(G)$, let us show that
  $E_{\theta} \embed C(\overline{G})$ densely: recall that, by
  lemma~\ref{lm:interp},
  \[
  E_\theta \embed \bigl( L^q,\dom(A) \bigr)_{\theta,\infty}
  \embed \bigl( L^q,\dom(A) \bigr)_{\theta',1}
  \embed \bigl[ L^q,\dom(A) \bigr]_{\theta'},
  \]
  where, by the characterization of $\dom(A)$ in
  \cite[theorem~4.9.1, p.~334]{TriebelI},
  \[
  \dom(A) = H_{q,D}^2(G) := \bigl\{
  \phi \in H_q^2(G):\; \phi\vert_{\partial G}=0 \bigr\}.
  \]
  Moreover, thanks to \cite[theorem~3.3.4, p.~321]{TriebelI}, one has
  \[
  \bigl[ L^q, H_{q,D}^2 \bigr]_{\theta'} = H_{q,D}^{2\theta'}
  \]
  if $2\theta' \neq 1/q$. Since $d>1$ and $2\theta' > d$ by
  hypothesis, the latter condition is obviously satisfied, hence
  $E_\theta \embed H_{q,D}^{2\theta'} \subset H_q^{2\theta'}$.
  Finally, the Sobolev embedding theorem (cf.~\cite[theorem~4.6.1,
  p.~328]{TriebelI}) yields
  $H_q^{2\theta'} \embed C(\overline{G})$, assuming that $2\theta' > d/q$,
  which is satisfied by hypothesis.
  We have thus shown that all assumptions of theorem~\ref{thm:astro}
  are met, hence the claim is proved.
\end{proof}
Note that $p > 2$ imply that, for $q$ large enough, the hypothesis
$d/(2q)< 1/2 - 1/p$ is always satisfied.

\begin{rmk}
  Instead of assuming that $f$ is an odd polynomial with negative
  leading coefficient, one could also assume that $f:\erre \to \erre$
  is locally Lipschitz continuous, polynomially bounded, and
  quasi-monotone, i.e. that there exists $\lambda>0$ such that
  $x \mapsto \lambda x - f(x)$ is increasing. In fact, assume that
  there exists $m>0$ such that $\abs{f(x)} \lesssim 1+\abs{x}^m$. By
  dissipativity of $f-\lambda I$,
  \[
    \ip[\big]{f(x+y)-\lambda(x+y) - (f(y)-\lambda y)}{x^*} \leq 0,
  \]
  hence
  \[
    \ip[\big]{f(x+y) - f(y)}{x^*} \leq \lambda \ip{x}{x^*} \leq \lambda,
  \]
  and
  \[
    \ip[\big]{f(x+y)}{x^*} \leq \lambda + \abs[\big]{\ip[\big]{f(y)}{x^*}}
    \lesssim \lambda + 1 + \norm{y}^m.
  \]
\end{rmk}

\subsection{Mild solution as random field}
We assume from now on, in addition to hypothesis~\ref{hyp:1}, the
following condition on the semigroup $S$ generated by $-A$.
\begin{hyp}
  \label{hyp:k}
  The semigroup $S=(S(t))_{t\geq 0}$ is sub-Markovian (i.e. $S(t)$ is
  positive and contracting in $L^\infty(G)$ for all $t \geq 0$) and
  admits a kernel, in the sense that there exists a function
  $K:\erre_+ \times G^2 \to \erre_+$ such that
  \[
    \bigl[ S(t)\phi \bigr](x) = \int_G K_t(x,y)\phi(y)\,dy
  \]
  for every $\phi \in L^q(G)$, $q \geq 1$.
\end{hyp}

Let $Q:=B B^*$, which is a symmetric and non-negative definite bounded
operator.  Recall that a cylindrical $Q$-Wiener process on
$L^2:=L^2(G)$ is a Gaussian family of random variables
$\mathcal W:=\{W_h(t),\, h\in L^2, t\geq 0\}$ such that, for all
$s,t\geq 0$ and $h,g\in L^2$, $\E(W_h(t))=0$ and
\[
 \E(W_h(t) W_g(s))=(t\land s)\langle Q h,g\rangle_{L^2}
\]
(in spite of the slight abuse of notation, no confusion should arise
with the cylindrical Wiener process $W$).  Let $L^2_Q$ be the Hilbert
space defined as the completion of $L^2$ with respect to the scalar
product $\langle h, g\rangle_{L^2_Q} := \langle Qh,
g\rangle_{L^2}$. Note that, denoting the pseudoinverse of $Q^{1/2}$ by
$Q^{-1/2}$, if $(e^k)_{k\in\enne}$ is a basis of $L^2$, then
$(\bar{e}^k):=(Q^{-1/2} e^k)$ is a basis of $L^2_Q$ . One can define
stochastic integrals with respect to $\mathcal{W}$ as follows (see,
e.g., \cite[Sec. 2]{Dalang-Quer}): let
$\{X(t,x): \, (t,x)\in [0,T]\times G\}$ be a predictable process in
$L^2(\Omega\times [0,T]; L^2_Q)$.  Then
\begin{equation}
  \label{eq:87}
  \int_0^T\!\!\int_G X(t,x) \,\mathcal{W}(dt,dx) :=
  \sum_{k=1}^\infty \int_0^T \ip[\big]{X(t,\cdot)}{\bar{e}^k}_{L^2_Q}
  \,dW_{\bar{e}^k}(t),
\end{equation}
and the isometry property reads
\[
  \E \abs[\bigg]{\int_0^T\!\!\int_G X(t,x) \,\mathcal{W}(dt,dx)}^2
  = \E \int_0^T \norm[\big]{X(t,\cdot)}^2_{L^2_Q}\,dt.
\]
In order to prove that the Malliavin derivative of the solution $u$ of
\eqref{eq:0} satisfies a stochastic equation, we need to verify that
$u$ can be interpreted as a mild solution to \eqref{eq:0} in the sense
of random fields (see, e.g.,
\cite{Dalang-EJP1999,Dalang-Quer,Walsh-LNM}). This is indeed the case
(cf. the analogous result for equations with additive noise in
\cite{Marinelli-Nualart-Quer}).
\begin{prop}
  \label{prop:99}
  Let the assumptions of proposition~\ref{prop:0} be satisfied. For
  any $(t,x) \in [0,T] \times G$, set $u(t,x):=[u(t)](x)$, where $u$
  is the unique $C(G)$-valued mild solution to \eqref{eq:0}.  Then
  for any $(t,x) \in \mathopen]0,T] \times G$,
  \begin{equation}
    \label{eq:60}
    \begin{split}
    u(t,x)
    & = \int_G K_t(x,y)u_0(y)\,dy
      + \int_0^t\!\!\int_G K_{t-s}(x,y)f(u(s,y))\,dy\,ds \nonumber \\
    &\quad + \int_0^t\!\!\int_G K_{t-s}(x,y) \sigma(u(s,y))\,\mathcal{W}(ds,dy).
    \end{split}
  \end{equation}
\end{prop}
\begin{proof}
  As in the proof of \cite[proposition 3.1]{Marinelli-Nualart-Quer},
  it suffices to show that, for every $t \in \mathopen]0,T]$ and for
  almost every $x \in G$, the process
  \[
  (s,y) \mapsto K_{t-s}(x,y)\sigma(u(s,y))
  \]
  belongs to $L^2(\Omega \times [0,T];L^2_Q)$ and that
  \begin{equation}
    \label{eq:86}
    \int_0^t S(t-s)\sigma(u(s))B\,dW(s) =
    \int_0^t \!\! \int_G K_{t-s}(\cdot,y) \sigma(u(s,y)) \mathcal{W}(ds,dy)
  \end{equation}
  as an equality in $L^2$.  Recalling that
  $(\bar{e}^k)=(Q^{-1/2} e^k)$, is a basis of the Hilbert space
  $L^2_Q$, one easily verifies that
  \[
    \norm[\big]{K_{t-s}(x,\cdot)\sigma(u(s,\cdot))}^2_{L^2_Q}
    = \sum_{k=1}^\infty \bigl( [S(t-s)\sigma(u(s))](\tilde{e}^k)(x) \bigr)^2,
  \]
  where $(\tilde{e}^k):=(Q^{1/2}e^k)$ is a basis of $Q^{1/2}(L^2)$. Note that
  \[
  \E \int_0^t \sum_{k=1}^\infty
  \norm[\big]{[S(t-s)\sigma(u(s))](\tilde{e}^k)}^2_{L^2}ds < \infty,
  \]
  because the stochastic integral on the left-hand side of
  \eqref{eq:86} is well defined. Thus, for almost all $x \in G$,
  \[
  \E \int_0^t \norm[\big]{K_{t-s}(x,\cdot)\sigma(u(s,\cdot))}^2_{L^2_Q}
  < \infty,
  \]
  so the stochastic integral on the right-hand side of \eqref{eq:86}
  is well defined. It remains equality in \eqref{eq:86}. Using the
  standard formal expansion of the cylindrical Wiener process $W$ as
  \[
  W(t)=\sum_{k=1}^\infty e^k w_k(t),
  \]
  where $w_k:=W_{\bar{e}^k}$, $k\geq 1$, form a family of independent
  standard one-dimensional Wiener processes, one has
  \[
    \int_0^t S(t-s)\sigma(u(s))B\,dW(s)
    = \sum_{k=1}^\infty \int_0^t\!\!\int_G
    K_{t-s}(\cdot,y)[\sigma(u(s))B e^k](y)\,dy\,dw_k(s).
  \]
  Then \eqref{eq:86} follows taking into account the definition
  \eqref{eq:87} and that $BB^*=Q$.
\end{proof}


\section{Equations with Lipschitz continuous coefficients}
\label{sec:Lip}
We assume throughout this section that the coefficients $f$ and
$\sigma$ in equation \eqref{eq:0} are Lipschitz continuous. We are
going to prove that, for any fixed $(t,x)\in (0,T]\times G$, the law
of the solution $u(t,x)$ to \eqref{eq:0} is absolutely continuous
with respect to the Lebesgue measure. For this, note that the
Gaussian space in which we will make use of the Malliavin calculus'
techniques is determined by the isonormal Gaussian process on the
Hilbert space $H := L^2(0,T;L^2_Q)$ which can be naturally
associated to the cylindrical $Q$-Wiener process $\mathcal{W}$
defined in the previous section (see \cite{nualart}).

We will first deal with the Malliavin differentiability of the
solution, and then we shall provide sufficient conditions implying
that the pointwise Malliavin derivative is non-degenerate.

We need further assumptions, that will be assumed to hold from now on.
\begin{hyp}
  \label{hyp:eta}
  One has
  \[
    \frac{d}{2q} < \frac12 - \frac1p.
  \]
  Moreover, $B \in \gamma(L^2,L^q)$ and $u_0 \in C(\overline{G})$.
\end{hyp}
\begin{hyp}
  \label{hyp:sa}
  The semigroup $S$ is self-adjoint and Markovian.
\end{hyp}
Recall also that we assume that hypotheses \ref{hyp:1} and \ref{hyp:k}
are in force throughout. By proposition~\ref{prop:0}, it follows that
\eqref{eq:0} admits a unique $C(\overline{G})$-valued mild solution
$u$, and that \eqref{eq:0} can also be written as an equality of
random fields.

\subsection{Pointwise Malliavin differentiability of the solution}
The main result of this section is the following.
\begin{thm}
  \label{thm:diff}
  Let $u \in L^p(\Omega;C([0,T];C(\overline{G})))$ be the unique mild
  solution to \eqref{eq:0}. Then
  \[
  u \in L^\infty([0,T] \times G; \bD^{1,\infty})
  \]
  and  the family of Malliavin derivatives
  $\{Du(t,x)\}_{(t,x)\in [0,T] \times G}$ satisfies the following
  linear equation in $H$:
  \begin{equation}
    \label{eq:diff}
    \begin{split}
      Du(t,x) &= v_0(t,x)
      + \int_0^t\!\!\int_G K_{t-s}(x,y) F(s,y) Du(s,y)\,dy\,ds \\
      &\quad + \int_0^t\!\!\int_G K_{t-s}(x,y) \Sigma(s,y)
      Du(s,y)\,\mathcal{W}(ds,dy)
    \end{split}
  \end{equation}
  where
  \[
  v_0(t,x):=(\tau,z) \mapsto K_{t-\tau}(x,z) \sigma(u(\tau,z))
  \,1_{[0,t]}(\tau),
  \]
  and $F$, $\Sigma \colon \Omega \times [0,T] \times G \to \erre$ are
  adapted bounded random fields.
\end{thm}

\noindent The stochastic integral in \eqref{eq:diff} must be
interpreted as an $H$-valued integral with respect to the cylindrical
$Q$-Wiener process $\mathcal{W}$ (see, e.g.,
\cite[\S3]{Nualart-QuerPOTA}).

\medskip

The following estimate plays an important role in the proof
theorem~\ref{thm:diff} as well as in several further developments. We
shall write $E^q_\eta$, for any $q \geq 1$ and $\eta>0$, to denote
$(I+A)^{-\eta}L^q$.
\begin{lemma}
  \label{lm:gemma}
  Let $v \in L^p(\Omega;C([0,T];C(\overline{G})))$ be adapted and
  $w:\Omega \times [0,T] \times G \to H$ be the process defined as
  \[
    w(t,x):=(\tau,z) \longmapsto K_{t-\tau}(x,z) \sigma(v(\tau,z))
    \,1_{[0,t]}(\tau).
  \]
  For any $\eta \in \mathopen] d/(2q),1/2-1/p \mathclose[$ one has
  \[
    \sup_{x \in G} \norm[\big]{w(t,x)}^2_H \lesssim
    \int_0^t \norm[\big]{S(t-s)\sigma(v(s))B}^2_{\gamma(L^2,E^q_\eta)}\,ds.
  \]
\end{lemma}
\begin{proof}
  Since $H=L^2(0,T;L^2_Q)$ and $\langle{Qh},{h}\rangle = \norm{B^*h}^2_{L^2}$ for
  every $h \in L^2_Q$, denoting a complete orthonormal basis of $L^2$
  by $(e^k)_{k\geq 1}$, it follows by Plancherel's theorem that
  \begin{align*}
    \norm[\big]{w(t,x)}^2_H
    &= \int_0^t \norm[\big]{%
      B^* K_{t-\tau}(x,\cdot)\sigma(v(\tau,\cdot))}^2_{L^2}\,d\tau\\
    &= \int_0^t \sum_{k\in\enne}
      \langle{K_{t-\tau}(x,\cdot) \sigma(v(\tau,\cdot))},{Be^k}\rangle^2\,d\tau\\
    &= \int_0^t \sum_{k \in \enne} \biggl(%
      \int_G K_{t-\tau}(x,z)\sigma(v(\tau,z))[Be^k](z)\,dz \biggr)^2 d\tau\\
    &= \int_0^t \sum_{k\in\enne}
      \bigl[S(t-\tau) \sigma(v(\tau)) Be^k\bigr](x)^2\,d\tau,
  \end{align*}
  where we have used the integral representation of the semigroup $S$
  in the last step. Let $(\gamma_k)$ be a sequence of independent
  standard Gaussian random variable on an auxiliary probability space
  $\Omega'$. Then
  \[
    \norm[\big]{w(t,x)}^2_H = \int_0^t \E'\abs[\bigg]{\sum_{k\in\enne}
    \gamma_k \bigl[S(t-\tau) \sigma(v(\tau)) Be^k\bigr](x)}^2\,d\tau,
  \]
  hence also, by Minkowski's inequality and the embedding
  $E^q_\eta \embed L^\infty$,
  \begin{align*}
    \sup_{x \in G} \norm[\big]{w(t,\cdot)}^2_{H}
    &\lesssim \int_0^t \E'\norm[\bigg]{\sum_{k\in\enne} \gamma_k
      \bigl[S(t-\tau) \sigma(v(\tau)) Be^k\bigr]}_{E^q_\eta}^2\,d\tau\\
    &= \int_0^t \norm[\big]{S(t-\tau) \sigma(v(\tau)) B}^2_{\gamma(L^2,E^q_\eta)}
      \,d\tau.
      \qedhere
  \end{align*}
\end{proof}

The proof of theorem~\ref{thm:diff} uses a maximal inequality for
stochastic convolutions, that is a special (simpler) case of
\cite[proposition~4.2]{vNVW}. We shall use the notation $R \diamond F$
to denote the process
\[
R \diamond F \colon t \mapsto \int_0^t R(t-s) F(s)\,dW(s),
\]
where $R$ is an analytic semigroup of contractions on a UMD Banach
space $E$ and $F \colon \Omega \times \erre_+ \to \cL(L^2,E)$ is an
$L^2$-strongly measurable and adapted process. Denoting the generator
of $R$ by $-C$, we shall write $E_\eta$, for any $\eta>0$, to denote
$\dom((I+C)^\eta)$.
\begin{prop}
  \label{prop:sc}
  Let $\alpha \in \mathopen]0,1/2\mathclose[$, $p>2$, $\theta \geq 0$
  be such that
  \[
  \theta < \alpha - \frac1p,
  \]
  and $T>0$. There exists $\varepsilon>0$ such that
  \[
  \E\norm[\big]{R \diamond F}^p_{C([0,T];E_\eta)} \lesssim
  T^{p\varepsilon} \int_0^T
  \E\norm[\big]{s \mapsto (t-s)^{-\alpha}F(s)}^p_{\gamma(L^2(0,t;L^2),E)}.
  \]
\end{prop}

We shall also need a deep result by Pisier (see \cite[theorem~1.2 and
remark~1.8]{Pis:hol} as well as \cite[p.~5730]{Xu:Hinf}) on
vector-valued extensions of analytic semigroup, according to which
hypothesis~\ref{hyp:sa} implies that $(S(t) \otimes I_H)_{t \geq 0}$,
where $I$ denotes the identity of $H$, admits a (unique) extensions
from $L^q \otimes H$ to $L^q(H)$, denoted by $S_H$, which is again
analytic.
Let $A_H$ denote the negative generator of $S_H$ and
$(\lambda+A_H)_{\lambda>0}^{-1}$ its resolvent. The Laplace transform
identity
\[
(\lambda+A_H)^{-1} = \int_0^\infty e^{-\lambda t} S_H(t)\,dt
\]
implies that $(\lambda+A_H)^{-1}$ coincides with the unique continuous
linear extension of $(\lambda+A)^{-1} \otimes I_H$ to $L^q(H)$.
By hypothesis \ref{hyp:eta} there exists $\eta \in \mathopen]
d/(2q),1/2-1/p \mathclose[$ such that $\dom(A^\eta) \embed L^\infty$,
hence $(I+A)^{-\eta} \in \cL(L^q,L^\infty)$.  Since $(I+A)^{-\eta}$ is
positivity preserving by hypothesis~\ref{hyp:k}, $(I+A)^{-\eta}$
admits a unique extension to a continuous linear operator from
$L^q(H):=L^q(G;H)$ to $L^\infty(H):=L^\infty(G;H)$, with the same norm
(see, e.g., \cite[theorem~12.2]{Janson}). By the above, recalling
well-known expressions for fractional powers of closed operators (see,
e.g., \cite[\S2.6]{Pazy}), this extension coincides with
$(I+A_H)^{-\eta}$. Therefore, setting $E^q_\eta(H):=(I+A)^{-\eta}
L^q(H)$, we have $E^q_\eta(H) \embed L^\infty(H)$.

\medskip

\begin{proof}[Proof of theorem~\ref{thm:diff}]
  Let $\Phi$ be the fixed-point operator associated to equation
  \eqref{eq:0}, i.e.
  \[
  \Phi: v \longmapsto S(t)u_0 + \int_0^t S(t-s)f(v(s))\,ds
  + \int_0^t S(t-s) \sigma(v(s))B\,dW(s).
  \]
  It follows by the (the proof) of theorem~\ref{thm:astro} that the
  operator $\Phi$, or a suitable power of it, is a contractive
  endomorphism of $L^p(\Omega;C([0,T];C(\overline{G})))$. We are going
  to show that, for any $p>2$, there exists $T_0>0$, a positive
  constants $c<1$ depending on $T_0$, and a positive constant $N$
  depending on the $L^p(\Omega;C([0,T];C(\overline{G})))$
  norm of $v$, such that
  \begin{equation}
    \label{eq:dfv}
    \norm[\big]{D\Phi(v)}_{L^\infty([0,T_0] \times G;L^p(\Omega;H))} \leq N
    + c \norm[\big]{Dv}_{L^\infty(0,T_0;L^p(\Omega;L^\infty(G;H)))}.
  \end{equation}
  Let $v \in L^p(\Omega;C([0,T];C(\overline{G})))$ be such that $Dv
  \in L^\infty(0,T;L^p(\Omega;L^\infty(G;H)))$. Writing
  \begin{align*}
  \bigl[ \Phi(v) \bigr](t,x) &= \int_G K_t(x,y)u_0(y)\,dy
  + \int_0^t\!\!\int_G K_{t-s}(x,y) f(v(s,y))\,dy\,ds\\
  &\quad + \int_0^t\!\!\int_G K_{t-s}(x,y) \sigma(v(s,y))\,\mathcal{W}(dy,ds),
  \end{align*}
  well-known criteria of Malliavin calculus imply that the Malliavin
  derivatives of all terms on the right-hand side exist, so that
  $D\bigl[ \Phi(v) \bigr](t,x)$ can be written as the right-hand side
  of \eqref{eq:diff} with $u$ replaced by $v$. The proof of
  \eqref{eq:dfv} will be split in several steps, where each term
  appearing in the expression of $D\Phi(v)$ is estimated.

  \medskip

  \noindent\textsc{Step 1.} Let us set, for every
  $(t,x)$, $(\tau,z) \in [0,T] \times G$,
  \[
  w_0(t,x):=(\tau,z) \mapsto K_{t-\tau}(x,z) \sigma(v(\tau,z))
  \,1_{[0,t]}(\tau).
  \]
  Let $\eta \in
  \mathopen]d/(2q),1/2-1/p\mathclose[$. Lemma~\ref{lm:gemma} yields
  \[
    \norm[\big]{w_0(t,\cdot)}^2_{L^\infty(H)} \lesssim
    \int_0^t \norm[\big]{S(t-\tau) \sigma(v(\tau)) B}^2_{\gamma(L^2,E^q_\eta)}
      \,d\tau,
  \]
  where
  \[
    \norm[\big]{S(t-\tau) \sigma(v(\tau)) B}_{\gamma(L^2,E^q_\eta)}
    \lesssim (t-\tau)^{-\eta} \,
    \norm[\big]{\sigma(v)}_{C([0,T];C(\overline{G}))} \,
    \norm[\big]{B}_{\gamma(L^2,L^q)}.
  \]
  This implies
  \begin{align*}
    \E \norm[\big]{w_0}^p_{L^\infty([0,T] \times G;H)}
    &\lesssim \Bigl(1 + \E \norm[\big]{v}^p_{C([0,T];C(\overline{G}))} \Bigr)
      \norm[\big]{B}^p_{\gamma(L^2,L^q)} \sup_{t\leq T} \biggl( \int_0^t
      (t-\tau)^{-2\eta}\,d\tau \biggr)^{p/2}\\
    &\lesssim \Bigl(1 + \E \norm[\big]{v}^p_{C([0,T];C(\overline{G}))} \Bigr)
      \norm[\big]{B}^p_{\gamma(L^2,L^q)} \, T^{p(1-2\eta)/2},
  \end{align*}
  where the last term on the right-hand side is finite by assumption.

  \medskip

  \noindent
  \textsc{Step 2.}  Let $\alpha<1/2$ such that $\eta < \alpha-1/p$.
  Recalling that $E^q_\eta(H) \embed L^\infty(H)$, Minkowski's and
  Jensen's inequality yield
  \begin{align*}
    \norm[\bigg]{\int_0^t S(t-s) F(s) Dv(s) \,ds}^2_{L^\infty(H)}
    &\lesssim_T \int_0^t \norm[\big]{S(t-s) F(s) Dv(s)}^2_{E^q_\eta(H)}\,ds\\
    &\lesssim \int_0^t (t-s)^{-2\eta} \norm[\big]{Dv(s)}^2_{L^q(H)}\,ds.
  \end{align*}
  Since $\eta < \alpha - 1/p$ by assumption, we have
  $-2\eta > -2\alpha + 2/p$, hence
  $-2\eta = -2\alpha + 2/p + \varepsilon$, with $\varepsilon>0$. Then
  \begin{align*}
    \int_0^t (t-s)^{-2\eta} \norm[\big]{Dv(s)}^2_{L^q(H)}\,ds
    &= \int_0^t (t-s)^{-2\alpha} (t-s)^{2/p+\varepsilon}
      \norm[\big]{Dv(s)}^2_{L^q(H)}\,ds\\
    &\leq t^{2/p+\varepsilon} \int_0^t (t-s)^{-2\alpha}
      \norm[\big]{Dv(s)}^2_{L^q(H)}\,ds\\
    &\lesssim_T \int_0^t s^{-2\alpha} \norm[\big]{Dv(t-s)}^2_{L^q(H)}\,ds.
  \end{align*}
  As the measure $\mu$ on $[0,t]$ defined as
  \[
    \mu(ds) := \frac{1-2\alpha}{t^{1-2\alpha}}s^{-2\alpha}\,ds
  \]
  is a probability measure, it follows by Jensen's inequality that
  \begin{align*}
    \biggl( \int_0^t s^{-2\alpha} \norm[\big]{Dv(t-s)}^2_{L^q(H)}\,ds
    \biggr)^{p/2}
    &= \biggl( \frac{t^{1-2\alpha}}{1-2\alpha} \int_0^t
      \norm[\big]{Dv(t-s)}^2_{L^q(H)} \, \mu(ds) \biggr)^{p/2}\\
    &\lesssim t^{(1-2\alpha)p/2} \int_0^t \norm[\big]{Dv(t-s)}^p_{L^q(H)}
      \, \mu(ds)\\
    &\lesssim t^{(1-2\alpha)(p/2-1)} \int_0^t s^{-2\alpha}
      \norm[\big]{Dv(t-s)}^p_{L^q(H)}\,ds\\
    &\lesssim_T \int_0^t (t-s)^{-2\alpha} \norm[\big]{Dv(s)}^p_{L^q(H)}\,ds.
  \end{align*}
  Therefore
  \begin{align*}
    \E\norm[\bigg]{\int_0^t S(t-s) F(s) Dv(s) \,ds}^p_{L^\infty(H)}
    \lesssim_T \int_0^t (t-s)^{-2\alpha} \E\norm[\big]{Dv(s)}^p_{L^\infty(H)}\,ds.
  \end{align*}

  \medskip

  \noindent
  \textsc{Step 3.} Using again the continuous embedding
  $E^q_\eta(H) \embed L^\infty(H)$, we have
  \begin{align*}
    \sup_{x \in G} \E\norm[\big]{[S \diamond (\Sigma Dv B)](t,x)}_H^p
    &\leq \E\norm[\big]{S \diamond (\Sigma Dv B)}^p_{C([0,t];L^\infty(H))}\\
    &\lesssim \E\norm[\big]{S \diamond (\Sigma Dv B)}^p_{C([0,t];E^q_\eta(H))}\\
    &\lesssim_T \E\int_0^t \norm[\big]{%
      (\tau-\cdot)^{-\alpha} \Sigma\,Dv\,B}^p_{\gamma(L^2(0,\tau;L^2),L^q(H))}\,d\tau\\
    &\lesssim \int_0^t \E\norm[\big]{%
      (\tau-\cdot)^{-\alpha} \Sigma\,Dv\,B}^p_{L^2(0,\tau;\gamma(L^2,L^q(H)))}\,d\tau,
  \end{align*}
  where the third inequality follows by proposition \ref{prop:sc}, as
  $L^q(H)$ is a UMD Banach space and $\eta < \alpha-1/p$, and the
  fourth estimate follows by Fubini's theorem and the embedding
  \[
    L^2(0,\tau;\gamma(L^2,L^q(H))) \embed
    \gamma(L^2(0,\tau;L^2),L^q(H)),
  \]
  which holds because $L^q(H)$ has type 2.  Since $Dv(s) \in
  L^\infty(H)$ by assumption and $\Sigma \in L^\infty([0,T] \times G)$
  by the Lipschitz continuity of $\sigma$, it follows that
  \[
    \norm[\big]{(\tau-s)^{-\alpha} \Sigma(s) Dv(s)
      B}_{\gamma(L^2,L^q(H))} \leq (\tau-s)^{-\alpha}
    \norm[\big]{\Sigma}_{L^\infty_{t,x}} \norm[\big]{Dv(s)}_{L^\infty_x(H)}
    \norm[\big]{B}_{\gamma(L^2,L^q)},
  \]
  hence
  \[
    \norm[\big]{(\tau-\cdot)^{-\alpha}
      \Sigma\,Dv\,B}^2_{L^2(0,\tau;\gamma(L^2,L^q(H)))}
    \leq \norm[\big]{B}^2_{\gamma(L^2,L^q)} \norm[\big]{\Sigma}^2_{L^\infty_{t,x}}
    \int_0^\tau (\tau-s)^{-2\alpha} \norm[\big]{Dv(s)}^2_{L^\infty_x(H)}\,ds.
  \]
  Proceeding as in the previous step, we obtain
  \[
    \E\norm[\big]{(\tau-\cdot)^{-\alpha}
      \Sigma\,Dv\,B}^p_{L^2(0,\tau;\gamma(L^2,L^q(H)))} \lesssim_T
    \int_0^\tau (\tau-s)^{-2\alpha}
    \E\norm[\big]{Dv(s)}^p_{L^\infty_x(H)}\,ds,
  \]
  therefore, by Tonelli's theorem,
  \begin{align*}
    \sup_{x \in G} \E\norm[\big]{\bigl[S \diamond
    (\Sigma Dv B)](t,x)}_H^p
    &\lesssim_T \int_0^{t} \int_0^\tau (\tau-s)^{-2\alpha}
      \E\norm[\big]{Dv(s)}^p_{L^\infty(H)}\,ds\,d\tau\\
    &= \int_0^t \E\norm[\big]{Dv(s)}^p_{L^\infty(H)}
      \int_s^t (\tau-s)^{-2\alpha} \,d\tau\,ds,
  \end{align*}
  where
  \[
    \int_s^t (\tau-s)^{-2\alpha} \,d\tau = \int_0^{t-s} \tau^{-2\alpha}\,d\tau
    = \frac{1}{1-2\alpha} (t-s)^{1-2\alpha},
  \]
  hence
  \[
    \sup_{x \in G} \E\norm[\big]{\bigl[S \diamond
      (\Sigma Dv B)](t,x)}_H^p \lesssim_T \int_0^t (t-s)^{-2\alpha}
    \E\norm[\big]{Dv(s)}^p_{L^\infty(H)}\,ds.
  \]

  \medskip

  \noindent
  \textsc{Step 4.} Setting
  \begin{gather*}
    \phi(t) := \E\norm[\big]{Dv(t)}^p_{L^\infty(H)}, \qquad
    \psi(t) := \E\norm[\big]{D\Phi(v)(t)}^p_{L^\infty(H)},\\
    N := 1 + \E \norm[\big]{v}^p_{C([0,T];C(\overline{G}))},
  \end{gather*}
  the estimates in the previous steps can be written as
  \[
    \psi(t) \lesssim_T N + \int_0^t (t-s)^{-2\alpha} \phi(s)\,ds,
  \]
  hence, using the notation $h^*(s) := \sup_{r \leq s} \abs{h(r)}$ for
  any function $h \colon \erre \to \erre$ for which it makes sense,
  \[
    \psi(t) \lesssim_T N + \phi^*(t) \int_0^t (t-s)^{-2\alpha} \,ds
    = N + \frac{1}{1-2\alpha} t^{1-2\alpha} \phi^*(t),
  \]
  thus also
  \[
    \psi^*(t) \lesssim_T N + \frac{1}{1-2\alpha} t^{1-2\alpha}
    \phi^*(t),
  \]
  from which \eqref{eq:dfv} follows.

  Let $u_0$ be identified with the process equal to $u_0$ for all $t
  \in [0,T]$, which clearly belongs to
  $L^p(\Omega;C([0,T];C(\overline{G})))$ and is such that $Du_0 \in
  L^\infty(0,T;L^p(\Omega;L^\infty(G;H)))$, and introduce the sequence
  of processes $(u_n)$, $u_n:=\Phi(u_{n-1})$. Then $u_n$ converges to
  $u$ in $L^p(\Omega;C([0,T];C(\overline{G})))$, possibly along a
  subsequence of the type $(kn)$, with constant $k$ (if $\Phi$ is not
  a contraction, but $\Phi^k$ is). In particular, $(u_n)$ is bounded in
  $L^p(\Omega;C([0,T];C(\overline{G})))$. This in turn implies, thanks
  to \eqref{eq:dfv}, that $(Du_n)$ is bounded in $L^\infty([0,T_0]
  \times G;L^p(\Omega;H))$. Let us show that this actually implies
  that $(Du_n)$ is bounded in $L^\infty([0,T] \times
  G;L^p(\Omega;H))$. In fact, setting
  \[
    \phi_n(s) := \E\norm[\big]{Du_n(s)}^p_{L^\infty(H)}, \qquad \phi_0
    := 1 + \sup_{n \in \enne} \E
    \norm[\big]{u_n}^p_{C([0,T];C(\overline{G}))} < \infty,
  \]
  we have already shown that
  \[
    \phi_{n+1}(t) \lesssim_T \phi_0 + \int_0^t (t-s)^{-2\alpha} \phi_n(s)\,ds
    \qquad \forall t \in [0,T],
  \]
  and that $(\phi^*_n(T_0))_n$ is bounded. We now proceed by
  induction: assuming that $(\phi^*_n(jT_0))_n$ is bounded, let us
  show that $(\phi^*_n((j+1)T_0))_n$ is also bounded. Let
  $jT_0 < t \leq (j+1)T_0$. We have
  \begin{align*}
    \phi_{n+1}(t)
    &\lesssim_T \phi_0 + \int_0^t (t-s)^{-2\alpha} \phi_n(s)\,ds\\
    &= \phi_0 + \int_0^{jT_0} (t-s)^{-2\alpha} \phi_n(s)\,ds
      + \int_{jT_0}^t (t-s)^{-2\alpha} \phi_n(s)\,ds,
  \end{align*}
  where $t>jT_0$ implies $t-s>jT_0-s$ and
  $(t-s)^{-2\alpha} < (jT_0-s)^{-2\alpha}$, hence
  \[
    \int_0^{jT_0} (t-s)^{-2\alpha} \phi_n(s)\,ds
    < \int_0^{jT_0} (jT_0-s)^{-2\alpha} \phi_n(s)\,ds
    \leq \frac{(jT_0)^{1-2\alpha}}{1-2\alpha} \phi_n^*(T_0),
  \]
  so that
  \begin{align*}
    \phi_{n+1}(t)
    &\lesssim_T \phi_0 + \frac{(jT_0)^{1-2\alpha}}{1-2\alpha} \phi_n^*(jT_0)
      + \int_{jT_0}^t (t-s)^{-2\alpha} \phi_n(s)\,ds\\
    &\lesssim_T \phi_0 + \frac{(jT_0)^{1-2\alpha}}{1-2\alpha} \phi_n^*(jT_0)
    + \phi_n^*((j+1)T_0) \int_{jT_0}^t (t-s)^{-2\alpha} \,ds,
  \end{align*}
  where
  \[
    \int_{jT_0}^t (t-s)^{-2\alpha} \,ds = \int_0^{t-jT_0} s^{-2\alpha}\,ds
    \leq \int_0^{T_0} s^{-2\alpha}\,ds = \frac{T_0^{1-2\alpha}}{1-2\alpha}.
  \]
  This in turn implies, taking the supremum over $[0,(j+1)T_0]$,
  \[
    \phi_{n+1}^*((j+1)T_0) \lesssim_T \phi_0 +
    \frac{(jT_0)^{1-2\alpha}}{1-2\alpha} \phi_n^*(jT_0) +
    \frac{T_0^{1-2\alpha}}{1-2\alpha} \phi_{n}^*((j+1)T_0).
  \]
  Since $\phi^*_n(jT_0)$ is bounded uniformly with respect to $n$ by
  the inductive assumption, we deduce that $\phi_n^*((j+1)T_0)$ is
  bounded uniformly over $n$ as well, thus completing the inductive
  argument.  This implies, by a standard argument based on the closure
  of the Malliavin derivative, that
  $u \in L^\infty([0,T] \times G;\bD^{1,p})$.

  Finally, the equation for $Du$ follows immediately by
  differentiating equation \eqref{eq:60}.
\end{proof}

\subsection{Non-degeneracy of the Malliavin derivative}
\label{sec:matrix}
This section is devoted to study, for any fixed
$(t,x) \in \mathopen]0,T] \times G$, the norm of the Malliavin
derivative of $u(t,x)$. Together with the results of the previous
section, we will deduce the existence of the density for the law of
the random variable $u(t,x)$. Recall that throughout the section we
are assuming that $f$ and $\sigma$ are globally Lipschitz continuous
functions.

We will need an estimate for the norm of $Du(t,x)$ in
\[
H(a,b) := L^2(a,b;L^2_Q), \qquad 0 \leq a < b \leq T.
\]
\begin{prop}
  \label{prop:sta}
  Let $0\leq a<b\leq T$, $p>2$, and $\eta \in
  \mathopen]d/(2q),1/2-1/p\mathclose[$. There exists a positive
  constant $N$, independent of $a$ and $b$, such that
  \[
   \sup_{(t,x)\in [a,b]\times G} \E\norm[\big]{Du(t,x)}^p_{H(a,b)}
   \leq N (b-a)^{p(1/2-\eta)}.
  \]
\end{prop}
\begin{proof}
  Repeating the proof of theorem~\ref{thm:diff} with $H$
  replaced by $H(a,b)$, we get
  \[
   \sup_{(t,x) \in [0,T] \times G} \E\norm[\big]{Du(t,x)}^p_{H(a,b)}
   \leq N \sup_{(t,x) \in [0,T] \times G} \E\norm[\big]{v_0(t,x)}^p_{H(a,b)},
  \]
  and, by lemma~\ref{lm:gemma},
  \[
    \sup_{x \in G} \norm[\big]{v_0(t,x)}^2_{H(a,b)}
    \lesssim \int_a^{t \wedge b}
  \norm[\big]{S(t-s) \sigma(u(s)) B}^2_{\gamma(L^2,E_\eta)}\,ds,
  \]
  where
  \[
    \norm[\big]{S(t-s) \sigma(u(s)) B}_{\gamma(L^2,E_\eta)} \lesssim
    (t-s)^{-\eta} \bigl( 1 + \norm{u}_{C([0,T \times \overline{G})}
    \bigr) \norm{B}_{\gamma(L^2,L^q)}.
  \]
  Therefore
  \begin{align*}
    \sup_{(t,x) \in [0,T] \times G} \E\norm[\big]{Du(t,x)}^p_{H(a,b)}
    &\lesssim \bigl( 1 + \norm[\big]{u}^p_{L^p(\Omega;C([0,T \times \overline{G}))}
      \bigr) \sup_{t \leq b} \biggl( \int_a^t (t-s)^{-2\eta}\,ds \biggr)^{p/2}\\
    &\lesssim (b-a)^{(1-2\eta)p/2}.
    \qedhere
  \end{align*}
\end{proof}

\smallskip

In the next result, we establish sufficient conditions under which
the norm of the Malliavin derivative of $u(t,x)$ does not vanish,
almost surely.
\begin{prop}
  \label{prop:1}
  Assume that there exists a constant $c>0$ such that
  $\abs{\sigma(z)} \geq c$ for all $z\in \erre$ and that $Q$ is
  positivity preserving. Let $(t,x) \in \mathopen]0,T] \times G$,
  $\alpha \in \mathopen]0,1/2\mathclose[$, and
  $\eta \in \mathopen]d/(2q),\alpha-1/p\mathclose[$. If there exist
  $\beta \in \mathopen]0,1-\alpha-\eta]$ such that
  \begin{equation}
  \lim_{\delta \to 0} \frac{\delta^{\beta}}{%
    \norm[\big]{K(x,\cdot)}_{H(0,\delta)}} = 0,
    \label{eq:88}
  \end{equation}
  then $\norm{Du(t,x)}_H>0$ almost surely.
\end{prop}
\begin{proof}
  We are going to estimate $\P(\norm{Du(t,x)}_H \leq 1/n)$ for $n \in
  \enne$ and pass to the limit as $n \to \infty$.
  Let $\delta \in \mathopen]0,1\mathclose[$, and set, for compactness
  of notation, $H_\delta:= H(t-\delta,t)$. The obvious inequality
  $\norm{a+b} \geq \norm{a} - \norm{b}$ applied to the expression of
  $Du$ given by theorem~\ref{thm:diff} yields
  \[
    \norm[\big]{Du(t,x)}_H \geq \norm[\big]{v_0(t,x)}_{H_\delta}
    - \norm[\big]{S \ast (F Du)(t,x) +
      S \diamond (\Sigma Du B)(t,x)}_{H_\delta}.
  \]
  Hence, simplifying the notation a bit and denoting the second term
  within the norm on the right-hand side by $Y$,
  \[
  \P\bigl( \norm{Du} \leq 1/n \bigr) \leq
  \P\bigl( \norm{v_0} - \norm{Y} \leq 1/n \bigr) =
  \P\bigl( \norm{Y} \geq \norm{v_0} - 1/n \bigr).
  \]
  Since $Q$ as well as the semigroup $S$ is positivity preserving,
  hence $K$ is positive, and $\sigma: \erre \to \erre$ is continuous,
  we have
  \begin{align*}
    \norm[\big]{v_0(t,x)}^2_{H_\delta}
    &= \int_{t-\delta}^t \norm[\big]{%
    K_{t-s}(x,\cdot) \sigma(u(s,\cdot))}^2_{L^2_Q} \,ds\\
    &= \int_{t-\delta}^t \! \int_G K_{t-s}(x,y)
      \sigma(u(s,y)) \,Q [K_{t-s}(x,\cdot)
      \sigma(u(s,\cdot))](y)  \, dy\,ds\\
    &= \int_{t-\delta}^t \! \int_G K_{t-s}(x,y)
      \abs{\sigma(u(s,y))} \,Q [K_{t-s}(x,\cdot)
      \abs{\sigma(u(s,\cdot))}](y) \,dy\,ds\\
    &\geq c^2 \int_{t-\delta}^t \! \int_G K_{t-s}(x,y)
      Q [K_{t-s}(x,\cdot)](y)  \, dy\,ds\\
    &= c^2 \int_0^\delta \norm[\big]{K_s(x,\cdot)}^2_{L^2_Q}\,ds
    = c^2 \norm[\big]{K_\cdot(x,\cdot)}^2_{H(0,\delta)}.
  \end{align*}
  This implies that we can use Chebyshev's inequality to write, for
  $n$ sufficiently large,
  \begin{align*}
  \P\bigl( \norm{Du(t,x)}_H \leq 1/n \bigr)
  &\leq \P\bigl( \norm{Y}_{H_\delta} \geq c \norm{K_\cdot(x,\cdot)}_{H(0,\delta)}
  - 1/n \bigr)\\
  &\leq \frac{\E\norm{Y}^p_{H_\delta}}{\bigl(c \norm{K_\cdot(x,\cdot)}_{H(0,\delta)}
  - 1/n \bigr)^p},
  \end{align*}
  where, thanks to theorem~\ref{thm:diff} and
  proposition~\ref{prop:sta},
  \begin{align*}
    \E\norm{Y}^p_{H_\delta}
    &= \E\norm[\big]{S \ast (F Du)(t,x)
    + S \diamond (\Sigma Du B)(t,x)}^p_{H_\delta}\\
    &\lesssim \delta^{p(1/2-\alpha)}
      \norm[\big]{Du}^p_{L^\infty([0,T] \times G;L^p(\Omega;H_\delta))}\\
    &\lesssim \delta^{p(1-\alpha-\eta)}.
  \end{align*}
  Taking the limit as $n \to \infty$, we are left with
  \[
  \P\bigl( \norm{Du(t,x)}_H = 0 \bigr) \lesssim
  \biggl( \frac{\delta^{1-\alpha-\eta}}{\norm{K_\cdot(x,\cdot)}_{H(0,\delta)}}
  \biggr)^p
  \]
  Since this inequality holds for every
  $\delta \in \mathopen]0,1\mathclose[$, and the limit of the
  right-hand side as $\delta \to 0$ is zero by assumption, it follows
  that $\P\bigl( \norm{Du(t,x)}_H = 0 \bigr) = 0$.
\end{proof}

As an immediate consequence of the above result and of theorem
\ref{thm:diff} we obtain sufficient conditions for the pointwise
absolute continuity of the law of the mild solution to \eqref{eq:0},
thanks to well-known criteria of the Malliavin calculus (see, e.g.,
\cite[theorem~2.1.3]{nualart}).
\begin{thm}
  \label{thm:1}
  Let $u \in L^p(\Omega;C([0,T];C(\overline{G})))$ be the unique mild
  solution to equation \eqref{eq:0}, with $f$ and $\sigma$
  Lipschitz continuous and $u_0 \in C(\overline{G})$. Assume that
  there exists $c>0$ such that $\abs{\sigma(z)} \geq c >0$ for all
  $z \in \erre$ and $Q=B B^*$ is positivity preserving.  Let
  $(t,x) \in \mathopen]0,T] \times G$,
  $\alpha \in \mathopen]0,1/2\mathclose[$, and
  $\eta \in \mathopen]d/(2q),\alpha-1/p\mathclose[$. If there exist
  $\beta \in \mathopen]0,1-\alpha-\eta]$ such that \eqref{eq:88} is
  fulfilled, then the law of the random variable $u(t,x)$ is
  absolutely continuous with respect to Lebesgue measure.
\end{thm}

\begin{example}
  \label{ex:1}
  Assume that $A$ has compact resolvent in $L^2$. Since $A$ is
  accretive and self-adjoint, there exist an orthonormal basis
  $(e^k)_{k \in \enne}$ of $L^2$ and a sequence
  $(\lambda_k)_{k \in \enne} \geq 0$ such that $e^k \in \dom(A)$,
  $Ae^k=\lambda_k e^k$ and $\lim_{k \to
    \infty}\lambda_k=+\infty$. Moreover, let $B=(I+A)^{-m}$, with
  $m \in \enne$, and fix $(t,x)\in \mathopen]0,T] \times G$. Since
  $Q=(I+A)^{-2m}$, one has, for any $\delta\in \mathopen]0,1\mathclose[$,
  \begin{align*}
    \norm{K_\cdot(x,\cdot)}^2_{H(0,\delta)}
    & = \int_0^\delta \!\! \int_G K_s(x,y) [Q K_s(x,\cdot)](y)\,dy\,ds\\
    & = \int_0^\delta \sum_{k\geq 0} (1+\lambda_n)^{-2m}
        \langle{K_s(x,\cdot)},{e^k}\rangle_{L^2}^2 \,ds \\
    & = \int_0^\delta \sum_{k\geq 1} (1+\lambda_n)^{-2m} \,
        e^{-2s \lambda_n} |e^k(x)|^2 \,ds \\
    & = \frac12 \sum_{k\geq 1} (1+\lambda_n)^{-2m} \, \lambda_n^{-1}
        \,(1-e^{-2 \delta \lambda_n}) |e^k(x)|^2.
  \end{align*}
  Moreover, we have that
  \[
    1-e^{-2\delta \lambda_n} \geq \frac{2 \delta \lambda_n}{1+2 \delta
      \lambda_n} \geq \frac{2 \delta \lambda_n}{1+2 \lambda_n}.
  \]
  Hence
  \[
    \|K_\cdot(x,\cdot)\|^2_{H(0,\delta)} \geq \delta \; \sum_{k\geq 1}
    (1+\lambda_n)^{-2m} \, (1+2 \lambda_n)^{-1} \, |e^k(x)|^2.
  \]
  Assuming that $x \in G$ is such that there exists $k \in \enne$ for
  which $e^k(x) \neq 0$, the quantity
  \[
    C_x := \sum_{k\geq 1} (1+\lambda_n)^{-2m} \, (1+2 \lambda_n)^{-1} \,
    |e^k(x)|^2
  \]
  is strictly positive. Therefore we have
  $\norm{K_\cdot(x,\cdot)}^2_{H(0,\delta)}  \geq C_x \delta$, i.e.
  \[
    \frac{\delta^{1/2}}{\norm{K_\cdot(x,\cdot)}_{H(0,\delta)}}
    \leq C_x^{-1/2},
  \]
  which implies that condition \eqref{eq:88}, hence also the
  assumptions of theorem \ref{thm:1}, are satisfied if we can find
  $\alpha$ and $\eta$ such that $1-\alpha-\eta>1/2$. This is possible
  if $m$ is sufficiently large, so that $B \in \gamma(L^2,L^q)$ with
  $q$ large and $d/(2q)$ is smaller than, say, $1/4$.
\end{example}


\section{Reaction-diffusion equations}
\label{sec:loc}
Let us now consider equation \eqref{eq:0} in the general case,
i.e. assuming that $f\colon \erre \to \erre$ is an odd polynomial with
negative leading coefficient. As already observed, we could also
assume that $x \mapsto f(x) - \lambda x$ is decreasing for some
$\lambda \geq 0$, locally Lipschitz continuous, and with polynomial
growth.

Let $u_0 \in C(\overline{G})$, and
$u \in L^p(\Omega;C([0,T];C(\overline{G})))$ be the unique mild
solution to \eqref{eq:0}, the existence of which is guaranteed by
proposition \ref{prop:0}. For every $n \in \enne$, consider the
function $f_n \colon \erre \to \erre$ defined as
\[
f_n(x) = \begin{cases} f(x), & \abs{x} \leq n,\\
f(nx/\abs{x}), & \abs{x} > n.
\end{cases}
\]
Then $f_n$ is Lipschitz continuous, and the equation
\[ 
  du_n(t) + Au_n(t)\,dt = f_n(u_n(t))\,dt + \sigma(u(t)) B\,dW(t),
  \qquad u(0)=u_0,
\]
admits a unique mild solution
$u_n \in L^p(\Omega;C([0,T];C(\overline{G})))$. Moreover, by
construction of $u$ (see \cite{KvN2}), $u_n$ coincides with $u$ on the
stochastic interval $[\![0,T_n]\!]$, where the stopping time
$T_n$ is defined as
\[
  T_n := \inf \bigl\{ t \geq 0:\, \norm{u_n(t)}_{C(\overline{G})}
  \geq n \bigr\} \wedge T,
\]
and $\lim_{n \to \infty} T_n = T$ almost surely. In particular,
$u_n \to u$ in $L^r(\Omega;C([0,T];C(\overline{G})))$ for all
$r \in [1,p\mathclose[$.
Let $t \in \mathopen]0,T]$ be arbitrary but fixed and set, for every
$n \in \enne$,
\[
  \Omega_n := \bigl\{ \omega \in \Omega:\, t \leq T_n(\omega)
  \bigr\}.
\]
Since $(T_n)$ is a sequence of stopping times monotonically increasing
to $T$ as $n \to \infty$, $(\Omega_n)$ is a sequence in $\cF$
monotonically increasing to $\Omega$ as $n \to \infty$. Clearly $\{t\}
\times \Omega_n \subset [\![0,T_n]\!]$, hence $u(t)=u_n(t)$ on
$\Omega_n$, as an identity in $C(\overline{G})$. This implies that
$u(t,x) = u_n(t,x)$ on $\Omega_n$ for every $x \in G$. Moreover, as
$f_n$ is Lipschitz continuous, theorem \ref{thm:diff} implies that
$u_n(t,x) \in \bD^{1,p}$ for every $x \in G$, for all $p \geq 1$.  We
have thus shown that $u(t,x) \in \bD^{1,p}_{\mathrm{loc}}$, with
localizing sequence $(\Omega_n, u_n(t,x))$ (cf. \cite[{\S}III]{BouHir}
or \cite[\S1.3.5]{nualart}). This implies that $u(t,x)$ is Malliavin
differentiable, i.e. that there exists a random variable $Du(t,x)$,
independent of the chosen localizing sequence, such that
$Du(t,x)=Du_n(t,x)$ on $\Omega_n$.

We are now in the position to state and prove the main result of the
paper.
\begin{thm}
  \label{thm:main}
  Let $u \in L^p(\Omega;C([0,T];C(\overline{G})))$ be the unique mild
  solution to equation \eqref{eq:0} with initial datum $u_0 \in
  C(\overline{G})$. Assume that $Q=BB^*$ is positivity preserving and
  that there exists $c>0$ such that $\abs{\sigma(z)} \geq c >0$ for
  all $z\in \erre$.  Let $(t,x) \in \mathopen]0,T] \times G$, $\alpha
  \in \mathopen]0,1/2\mathclose[$, and $\eta \in
  \mathopen]d/(2q),\alpha-1/p\mathclose[$. If there exist $\beta \in
  \mathopen]0,1-\alpha-\eta]$ such that
  \[
   \lim_{\delta \to 0} \frac{\delta^{\beta}}{%
    \norm[\big]{K(x,\cdot)}_{H(0,\delta)}} = 0,
  \]
  then the law of the random variable $u(t,x)$ is absolutely
  continuous with respect to the Lebesgue measure on $\erre$.
\end{thm}
\begin{proof}
  Let $(t,x) \in G_T$ be arbitrary but fixed. Then, by the
  Bouleau-Hirsch' criterion (see \cite[proposition~7.1.4]{BouHir}), it
  suffices to prove that $\norm{Du(t,x)}_H > 0$ almost surely. Since
  $f_n$ is Lipschitz continuous for all $n \in \enne$,
  $\norm{Du(t,x)}_H >0$ on $\Omega_n$ for all $n \in \enne$. This
  readily implies that $\norm{Du(t,x)}_H >0$ almost surely: assume by
  contradiction that there exists $\Omega' \subset \Omega$ with
  strictly positive probability such that $\norm{Du(t,x)}_H =0$ on
  $\Omega'$. Since $\Omega_n$ increases monotonically to $\Omega$,
  there exists $n_0 \in \mathbb{N}$ such that
  $\P\left(\Omega''\right)>0$, where $\Omega'':=\Omega_{n_0} \cap
  \Omega'$. In particular, by definition of $\Omega_{n_0}$, one has
  $\norm{Du(t,x)}_H >0$ on $\Omega''$ because $\Omega'' \subset
  \Omega_{n_0}$. This is clearly a contradiction, because $\Omega''
  \subset \Omega'$. The claim is thus proved.
\end{proof}

\begin{rmk}
  Very minor adjustments allow to consider the case where
  $\sigma \colon \erre \to \erre$ is locally Lipschitz continuous with
  linear growth. In fact, the construction of a unique global solution
  is obtained again by \textit{r\'ecollement} of local solutions (see
  \cite{KvN2}), and the above reasoning can be repeated almost
  verbatim.
\end{rmk}

\begin{rmk}
  The setting of example \ref{ex:1} obviously satisfies the
  assumptions of theorem \ref{thm:main}.
\end{rmk}

\let\oldbibliography\thebibliography
\renewcommand{\thebibliography}[1]{%
  \oldbibliography{#1}%
  \setlength{\itemsep}{-1pt}%
}
\bibliographystyle{amsplain}
\bibliography{ref,lluis}

\end{document}